\documentclass{article}

\usepackage{amsmath}
\usepackage{amsthm}
\usepackage{thmtools}
\usepackage{amsfonts}
\usepackage{graphicx}
\usepackage{subfigure}
\usepackage{hyperref}
\usepackage{xspace}
\usepackage{parskip}
\usepackage{color}
\usepackage{tikz}
\usepackage{verbatim}
\usepackage{xspace}
\usepackage{mdframed}
\usepackage{mathrsfs}

\graphicspath{{figures/}}

\makeatletter
\def\thm@space@setup{%
  \thm@preskip=\parskip \thm@postskip=0pt
}
\makeatother

\declaretheorem[parent=section]{lemma}
\declaretheorem[sibling=lemma]{theorem}
\declaretheorem[sibling=lemma, name=Proposition]{prop}

\declaretheorem[sibling=lemma]{corollary}
\declaretheorem[sibling=lemma]{problem}

\newcommand{\boundary}{\partial}
\newcommand{\set}[1]{\left\{#1\right\}}

\newcommand{\closure}[1]{\overline{#1}}
\newcommand{\scc}{simple closed curve\xspace}

\newcommand{\abs}[1]{\left|#1\right|}
\newcommand{\defn}[1]{\emph{#1}}

\newcommand{\pd}{\partial}
\newcommand{\A}{\alpha}
\newcommand{\n}{\beta}

\renewcommand{\L}{\mathcal{L}}

\newcommand{\Rhat}{\widehat{R}}

\begin{document}
\title{Bridge spectra of twisted torus knots}
\author{R. Sean Bowman, Scott Taylor, and Alexander Zupan\thanks{The third author is supported by the National Science Foundation under Award No. DMS-1203988.}}

\date{}

\maketitle

\begin{abstract}
  We compute the genus zero bridge numbers and give lower bounds on
  the genus one bridge numbers for a large class of sufficiently
  generic hyperbolic twisted torus knots.  As a result, the bridge
  spectra of these knots have two gaps which can be chosen to be
  arbitrarily large, providing the first known examples of hyperbolic
  knots exhibiting this property.  In addition, we show that there are
  Berge and Dean knots with arbitrarily large genus one bridge
  numbers, and as a result, we give solutions to problems of
  Eudave-Mu\~noz concerning tunnel number one knots.
\end{abstract}

\section{Introduction}
The \defn{bridge spectrum} \cite{Zupan13} of a knot records the
minimum bridge numbers of a knot with respect to Heegaard surfaces of
all possible genera in a 3-manifold. In particular for a knot $K
\subset S^3$, the bridge spectrum is the sequence

\[ \mathbf{b}(K) = (b_0(K),b_1(K),b_2(K),\dots),\] 

where $b_g(K)$ is the minimal bridge number of $K$ with respect to a
genus $g$ Heegaard surface in $S^3$. The process of meridional
stabilization (see below) shows that, for all $g$ such that $b_g(K) > 0$,

\[ b_{g+1}(K) \leq b_{g}(K) - 1 \]

It is natural to consider the \defn{gaps} in the bridge spectrum: that
is, the values of $g$ for which $b_g(K) - b_{g+1}(K)$ (the
\defn{order} of the gap) is greater than one. Zupan \cite{Zupan13}
showed that for each $n$ there is an infinite family of iterated
torus knots with exactly $n$ gaps of arbitrarily large order. Those
examples are, of course, non-hyperbolic. Rieck and Zupan
\cite{JohnsonLDT,Zupan13} have asked if there are \emph{hyperbolic} knots
having more than one gap in their bridge spectrum. We answer the
question affirmatively by exhibiting a class of twisted torus knots
with two gaps of arbitrarily large order. We show:

\begin{theorem}\label{thm:hyperbolic-gaps}
  Given $C>0$, there are integers $p$, $q$, and $r$ with
 and $1<r<p<q$ such that whenever $|s| > 18p$, the twisted torus knot $K=T(p,q,r,s)$ is hyperbolic
  and satisfies
  \[ b_0(K) = p \]
  and
  \[ C\leq b_1(K)\leq\frac{1}{2}p. \]
  Moreover, these knots may be chosen to have tunnel number one.
\end{theorem}

Consequently, such knots have two gaps, of arbitrarily large order, in
their bridge spectra:
\begin{corollary}
  Given $C > 0$, there are hyperbolic knots in $K\subset S^3$ with
  \[ b_0(K)-b_1(K) \geq b_1(K) - b_2(K) \geq C. \] 
\end{corollary}

Twisted torus knots are an interesting class of knots which have been
studied in several contexts.  These knots represent many
different knot types.  Morimoto and Yamada~\cite{MorimotoYamada09} and
Lee~\cite{Lee12} have constructed twisted torus knots which are
cables.  Morimoto has also shown that infinitely many twisted torus
knots are composite~\cite{Morimoto12}.  Guntel has shown that
infinitely many twisted torus knots are torus knots~\cite{Guntel12}.
Lee~\cite{Lee13} has characterized twisted torus knots which are
actually the unknot. Moriah and Sedgwick have shown that certain
hyperbolic twisted torus knots have minimal genus Heegaard splittings
which are unique up to isotopy~\cite{MoriahSedgwick09}.  It is also
known that twisted torus knots have arbitrarily large
volume~\cite{FuterEtAl11}. Little, however, has been proved about
the bridge numbers of twisted torus knots apart from Morimoto, Sakuma,
and Yokota's result that a certain infinite family is not one bridge
with respect to a genus one splitting of the
$3$--sphere~\cite{MorimotoSakumaYokota96}.

Twisted torus knots are also interesting from the point of view of
Dehn surgery.  Berge constructed examples of knots in $S^3$ with lens
space surgeries, infinitely many of which are twisted torus
knots~\cite{Berge90}.  Later, Dean constructed twisted torus knots with
small Seifert fibered surgeries~\cite{Dean03}.  

The knots constructed by Berge are knots lying on a genus two
splitting of $S^3$ which represent a primitive element in the
fundamental group of the handlebody on either side.  Such knots are
called doubly primitive, and it is not difficult to see that surgery
along the slope determined by the splitting yields a lens space.  An
open question is whether the list Berge gives in~\cite{Berge90} is a
complete list of all knots in $S^3$ with lens space surgeries (problem
1.78 in Kirby's list~\cite{Kirby10}).  Many of the knots in Berge's list
have bounded genus one bridge number.  However, we show the following
theorem:

\begin{theorem}\label{thm:hyperbolic-berge}
  There are hyperbolic Berge knots with arbitrarily large genus one
  bridge number.  These are Berge knots of type VII and VIII, knots
  which lie in the fiber of a genus one fibered knot in $S^3$.
\end{theorem}

A similar result holds for the Dean knots mentioned above:

\begin{theorem}\label{thm:hyperbolic-dean}
  There are hyperbolic Dean knots with arbitrarily large genus one bridge
  number. 
\end{theorem}

We remark that~\autoref*{thm:hyperbolic-berge} has been known to Ken Baker
and Jesse Johnson for some
time~\cite{BakerCommunication,JohnsonCommunication}.

Finally, we note that in Problems 2.1 and 2.3 of~\cite{Gordon05},
Eudave-Mu\~noz proposed the following:

\begin{problem}\label{problem1}
  Give explicit examples of tunnel number one knots with arbitrarily
  large genus one bridge number.
\end{problem}
\begin{problem}\label{problem2}
  Give explicit examples of tunnel number one knots $K$ with arbitrarily
  large genus one bridge number such that, in addition, a minimal genus
  Heegaard surface for the exterior $M_K$ has Hempel distance two.
\end{problem}
Relevant definitions may be found in~\cite{Gordon05}.  As a corollary
to Theorem 1.3 we provide a solution to these problems: 

\begin{corollary}\label{cor:distance-two}
For any $m > 1$, the family of Berge knots $K^n = T(mn+1,mn + n + 1,n, \pm 1)$ has the property that for sufficiently large $n$, $M_{K^n}$ has a minimal genus Heegaard surface of distance two, and
\[ \lim_{n \rightarrow \infty} b_1(K^n) = \infty.\]
\end{corollary}


The plan of the paper is as follows:
In~\autoref{section:preliminaries}, we introduce relevant terminology
and important results used in the rest of the paper.  Next, we
determine the genus zero bridge numbers for twisted torus knots with
certain parameters in~\autoref{section:genus-zero}.
In~\autoref{section:construction}, we construct collections of twisted
torus knots and establish properties which will be used to prove the
main theorems in~\autoref{section:main}.

\textbf{Acknowledgments}  The authors would like to thank Jesse
Johnson and Ken Baker for sharing their knowledge of bridge numbers of
Berge knots.  We would also like to thank Yo'av Rieck for helpful
conversations.  Finally, the third author would like to thank Colby College for its hospitality.

\section{Preliminaries}\label{section:preliminaries}
For convenience, we work in $M = S^3$, although some of our results
hold in more general 3-manifolds. For a link $L \subset M$, we
denote the exterior $M_L =\closure{M\setminus N(L)}$ of $L$ in $M$ by
$M_L$ (and use $N(L)$ to denote an closed regular neighborhood of $L$
in $M$).

A collection of arcs $\tau$ properly embedded in a handlebody $H$ is
\emph{trivial} if each arc $t \in \tau$ cobounds a disk $\Delta$ with
an arc in $\pd H$ such that $\Delta \cap \tau = t$.  We call such
$\Delta$ a \emph{bridge disk}.  A \emph{$(g,b)$-bridge splitting} (or
$(g,b)$-splitting) is the decomposition of $(M,L)$ as $(V,\A)
\cup_{\Sigma} (W,\n)$, where $V$ and $W$ are genus $g$
handlebodies containing nonempty collections $\A$ and $\n$ of $b$
trivial arcs, $M = V \cup_{\Sigma} W$ is a Heegaard splitting, and $L
= \A \cup \n$.  We call the closed surface $\Sigma$ a
\emph{$(g,b)$-bridge surface} and let $\Sigma_L$ denote $\Sigma \cap
M_L$. The surface $\Sigma_L$ has $2b$ meridional boundary
components. It is well known that for any link $L \subset M$ and any
genus $g$, $(M,L)$ admits a $(g,b)$-bridge splitting for some $b$.
Thus, for each $g$ the link $L$ has a \emph{genus $g$ bridge number}
$b_g(L)$, where
\[ b_g(L) = \min\{b: L \text{ admits a $(g,b)$-splitting}\}\]
These invariants, first defined by Doll in \cite{Doll92}, generalize
the classical bridge number $b(L) = b_0(L)$ due to Schubert
\cite{Schubert54}.   

Given a $(g,b)$-bridge surface $\Sigma$ for a link $L$, we can create
a $(g+1,b-1)$-bridge surface for $L$ by tubing $\Sigma$ along any one
of the arcs $L \setminus \Sigma$. Thus, $b_{g+1}(L) \leq b_g(L) - 1$. This
process is called \defn{meridional stabilization} and is described in
1detail in \cite[Lemma 3.2]{ScharlemannTomova08}.

Let $S$ be a surface properly embedded in $M_L$.  A \scc in $S$ is
called $\defn{essential}$ if it is not parallel to a component of
$\boundary S$ and does not bound a disk in $S$.  A properly embedded
arc in $S$ is called \defn{essential} if it does not cobound a disk in $S$
with an arc in $\boundary S$.  A \emph{compressing disk} $D$ for $S$ is
an embedded disk such that $D \cap S = \pd D$ and $\pd D$ is an
essential curve in $S$.  A \emph{$\pd$-compressing disk} is an
embedded disk $\Delta$ such that $\Delta \cap S$ is an arc $\gamma$
which is essential in $S$, $\Delta \cap \pd M_L$ is an arc $\delta$,
and $\pd \Delta$ is the endpoint union of $\gamma$ and $\delta$.  We
say that $S$ is \emph{incompressible} if there does not exist a
compressing disk $D$ for $S$.  An incompressible surface $S$ is said
to be \emph{essential} if it is not $\boundary$--parallel and, in
addition, there is no boundary compressing disk for $S$.  As $\pd M_L$
is a collection of tori, $S$ is essential if and only if it is
incompressible and is not a $\pd$-parallel annulus or torus.

Suppose that $\Sigma$ is an $n$-bridge sphere for a link $L$ in $M=S^3$, and suppose further that there exist bridge disks $\Delta_1$ and $\Delta_2$ on opposite sides of $\Sigma$ with the property that $\Delta_1 \cap \Delta_2$ is one or two points contained in $L$.  If $|\Delta_1 \cap \Delta_2| = 1$, then we may reduce the number of bridges of $L$ with respect to $\Sigma$, and we say $\Sigma$ is \defn{perturbed}.  On the other hand, if $|\Delta_1 \cap \Delta_2| = 2$, we say $\Sigma$ is \defn{cancellable}.  In this case, there is a
component $L' \subset L$ contained in $\boundary \Delta_1 \cup
\boundary \Delta_2$, and the discs $\Delta_1 \cup \Delta_2$ may be used to
isotope $L'$, in the complement of $L \setminus L'$, into
$\Sigma$. Alternatively, we may isotope $\Sigma$ along $\Delta_1 \cup
\Delta_2$ to a bridge surface $\Sigma'$ for $L \setminus L'$
containing $L'$. We call $\Sigma'$ the \defn{result of cancelling
  $\Sigma$}.

Suppose now that $S\subset M$ is a connected, separating surface
which intersects $L$ transversely, and $S_L = \overline{S \setminus
  N(L)}$ is either compressible or boundary-compressible to both sides
in $M_L$.  Suppose further that for any pair of compressing or
boundary-compressing disks $D$ and $D'$ on opposite sides of $S_L$, we
have $D \cap D' \neq \emptyset$.  In this case we call $S$
\defn{strongly irreducible}.  The following result (a special case of
a theorem of Hayashi and Shimokawa) reveals the significance of strongly irreducible surfaces.

\begin{theorem}[\cite{HayashiShimokawa01}, Theorem
  1.2]\label{weakly-incompressible} 
  Suppose that $L$ is a link in $S^3$ and $\Sigma$ is a bridge sphere
  for $L$.  Then $\Sigma$ is perturbed, cancellable, or for each
  component $L'$ of $L$, there is strongly irreducible 2-sphere
  $\Sigma'$ which meets $L'$ and which intersects every component of
  $L$ at most as many times as $\Sigma$.
\end{theorem}

It is well known that any two essential surfaces in an irreducible
link exterior may be isotoped so that all arcs of intersection are
essential in both surfaces.  This property extends to a strongly
irreducible surface and an essential surface. This result is also
likely well known. A proof can be found in \cite[Proposition
6.1]{BlairEtAl13}; it is also implicit in \cite[Lemma 5.2]{Zupan13}.

\begin{lemma}\label{essential-intersection}
  Suppose $M_L$ is an irreducible link exterior containing an essential surface $S$ and a strongly irreducible surface $S'$.  There exists
  an isotopy after which all arcs of $S \cap S'$ are essential in both
  surfaces and $|\boundary S \cap \boundary S'|$ is minimal
  up to isotopy.
\end{lemma}


Once we know that all arcs of intersection of two surfaces are
essential in both of them, we can utilize the next lemma, which is
based on similar results proved by
Gordon-Litherland~\cite{GordonLitherland84}, Rieck~\cite{Rieck00}, and
Torisu~\cite{Torisu96}.

\begin{lemma}\label{arc-annulus-bound}
  Let $L \subset M$ be a link such that $M_L$ is irreducible, $\boundary$-irreducible, and
  anannular. Suppose that $F$ and $G$ are connected,
  orientable surfaces with nonempty boundary properly embedded in
  $M_L$ such that $\chi(F),\chi(G) < 0$, $G$ is essential in $M_L$,
  $|\partial F \cap \partial G|$ is minimal up to isotopy, and $F$ and
  $G$ intersect in a nonempty collection of $n$ essential arcs.  Then
  \[ n\leq 9\chi(F)\chi(G). \]
\end{lemma}
\begin{proof}
  Let $S$ be a connected, orientable surface properly embedded in
  $M_{L}$ such that $\chi(S)<0$.  Let $\Lambda$ be a collection of
  properly embedded essential arcs in $S$ such that no two are
  parallel in $S$.  Then $\Lambda$ can be completed to an ideal
  triangulation of the interior of $S$ by adding more edges between
  components of $\boundary S$ if necessary.  Let the new collection of
  edges be $\Lambda'$ and the set of faces of the ideal triangulation
  be $F$.  Then we have $3\abs{F}=2\abs{\Lambda'}$ as well as
  $\chi(S)=-\abs{\Lambda'}+{F}$, and so
  $\abs{\Lambda}\leq\abs{\Lambda'}\leq -3\chi(S)$.


  Viewing the intersection $F \cap G$ as a graph
  $\Lambda_F \subset F$, let $\Lambda_F'$ be the reduced graph
  (obtained by combining all sets of parallel edges into a single
  edge), and let $m_F$ be the maximal number of mutually parallel
  edges in $\Lambda_F$, so that each edge in $\Lambda_F'$ corresponds
  to at most $m_F$ edges of $\Lambda_F$.  As $\Lambda'_F$ has no
  parallel edges,
  \[   n/m_F \leq -3\chi(F).   \] 
  and so
  \[   n/(-3\chi(F)) \leq m_F.   \]
  Let $\Lambda_G$ be the graph in $G$ consisting of a collection of
  $m_F$ arcs of $F \cap G$ which are mutually parallel in $F$.  If
  $\Lambda_G$ has no parallel edges, then $m_F \leq -3\chi(G)$; hence
  \[   n \leq 9\chi(F)\chi(G).   \]

  Otherwise, there are arcs $\lambda_1, \lambda_2 \subset F \cap G$
  which are parallel in both $F$ and $G$, chosen to be adjacent in
  $G$.  Then $\lambda_1$ and $\lambda_2$ cobound rectangles $R_F
  \subset F$ and $R_G \subset G$ with arcs in $\partial F$ and
  $\partial G$, and we have that $A = R_F \cup R_G$ is a properly
  embedded annulus or M\"obius band.  Suppose first that $A$ is a
  M\"obius band.  If $\partial A$ is inessential in $\partial M_L$,
  then $M_L$ contains a properly embedded $\mathbb{RP}^2$,
  contradicting that $M_L$ is irreducible.  The case in which
  $\partial A$ is essential is ruled out by Lemma 5.1
  of~\cite{Rieck00}.  Next, suppose that $A$ is an annulus.  If one
  component of $\partial A$ is inessential, then compressing along
  this component yields an embedded disk; hence, the other component
  of $\partial A$ must be inessential as $\boundary M_L$ is
  incompressible.  In this case, $R_F$ is isotopic to $R_G$ in $M_L$.
  However, such an isotopy would allow us to reduce the number of arcs
  of $F \cap G$, contradicting the minimality of $\partial F
  \cap \partial G$.  Thus, we may assume that both components of
  $\partial A$ are essential.  As such, $A$ is incompressible.  If $A$
  is boundary parallel, then $\lambda_1$ cobounds a boundary
  compressing disk for $G$, which contradicts the assumption that $G$
  is essential.  It follows that $A$ is an essential annulus,
  contradicting the assumption that $M_L$ is anannular.
\end{proof}

We make one more definition before proceeding.  Let $K$ be a knot in
$M=S^3$.  Then $H_1(\pd N(K))$ has a natural basis
$([\mu],[\lambda])$, where $\mu$ bounds a meridian disk of $N(K)$ and
$[\lambda] = 0$ in $H_1(M_K)$.  We parameterize a given curve $\gamma
\subset \pd M_K$ as a fraction (or \defn{slope}) $\frac{a}{b}$, where
$[\gamma] = a[\mu] + b[\lambda]$.  Given such a $\gamma$, we may
construct a new manifold $M_K(\gamma)$ by gluing a solid torus $V$ to
$\pd M_K$ so that a curve bounding a meridian disk of $V$ is glued to
$\gamma$.  We say that the resulting manifold $M_K(\gamma)$ is the
result of \defn{$\gamma$ Dehn surgery} on $K$.

\section{Genus zero bridge numbers of some twisted torus knots}\label{section:genus-zero}
To begin this section, we define the class of knots known as twisted
torus knots.  Assume that $p,q>1$ are relatively prime integers.
Consider a $(p,q)$ torus knot $T_{p,q}$ which lies on a Heegaard torus
$T$ for $S^3$.  Let $C$ be a curve bounding a disk $D$ such that $D$
meets $T$ in a single arc and the interior of $D$ meets $K$ in $0\leq
r\leq p+q$ points of the same sign.  We say that the result of doing
$-1/s$ Dehn surgery on $C$ is the \defn{twisted torus knot} $T(p, q,
r, s)$.  Informally, this is the knot obtained from $T_{p,q}$ by
twisting $r$ parallel strands by $s$ full twists.  Note that we leave
open the possibility that $r\leq 1$, in which case the resulting knot
is clearly a torus knot.  Note also that various alternate definitions
of twisted torus knots exist in the literature.  For a discussion of
these variations, see (for instance) \cite{FuterEtAl11}.

For the remainder of this section, we let $K$ denote $T_{p,q}$.  If
the link $K\cup C$ is hyperbolic, we know from results of Thurston
that the resulting knot will be hyperbolic for all but finitely many
surgeries.  The following proposition of Lee says that, for most
values of $p$, $q$, and $r$, this is indeed the case.

\begin{prop}[\cite{Lee13}, Proposition 5.7]\label{hyperbolic-twists}
  When $r>1$ is not a multiple of $p$ or $q$, the link $K \cup C$
  is hyperbolic.
\end{prop}

Next, we exhibit an incompressible surface in $M_{K \cup C}$ which
will play the role of $G$ in~\autoref*{arc-annulus-bound} above.
Since $K$ is a torus knot, $M_K$ contains an essential annulus $G'$
which intersects $C$ transversely in two points.  We let $G$ denote
$G' \cap M_{K \cup C}$, so that $G$ may be regarded at a
twice-punctured annulus (or a 4-punctured sphere).

\begin{lemma}\label{four-punctured-sphere}
The surface $G$ is essential in $M_{K \cup C}$.
\end{lemma}
\begin{proof}
  Since $G$ is a 4-punctured sphere, it suffices to show that $G$ is
  incompressible.  Suppose by way of contradiction that $\gamma
  \subset G$ bounds a compressing disk $D$ in $M_{K \cup C}$.  As a
  curve in the annulus $G'$, $\gamma$ must be inessential since $K$ is
  a nontrivial torus knot.  Thus, $\gamma$ bounds a twice-punctured
  disk in $G$, and compressing $G$ along $D$ yields an essential annulus contained in $M_{K \cup C}$,
  contradicting~\autoref*{hyperbolic-twists}.
\end{proof}

In the link manifold $M_{K \cup C}$, let $T_K = \pd N(K)$ and $T_C =
\pd N(C)$.  In addition, for each $s$, let $K_s=T(p,q,r,s)$ and $L_s =
K_s\cup C_s$, where $C_s$ is the core of the solid torus which results
from performing $1/s$ surgery on the twisting curve $C$.  With this
notation, $K = K_0$ and $C = C_0$.  In the following lemma, we compare
the $b_0(K_s)$ to $b_0(K)$, where
\[ b_0(K) = \min\{p,q\}\]
by a theorem of Schubert~\cite{Schubert54} with a modern proof given by
Schultens~\cite{Schultens07}.

\begin{lemma}\label{yoav}
  Given $p$, $q$, and $r$ satisfying $1<r < p < q$, if $|s| > 18p$, we
  have
  \[ b_0(K_s)= p. \]
\end{lemma}
\begin{proof}
  The torus knot $K$ lies on a Heegaard torus $T$ for $S^3$. We may
  consider $T$ as the boundary of a regular neighborhood of an unknot
  in 1-bridge presentation with respect to a bridge sphere $\Sigma$
  for $S^3$. In fact, we may arrange for $\Sigma$ to be a $p$-bridge
  sphere for $K$. The bridge sphere $\Sigma$ then realizes the minimal
  bridge number $b_0(K) = p$ of $K$.  Since $r < p$, it is not hard
  (though somewhat tedious) to show that $C$ may be isotoped in the
  exterior of $K$ to lie on $\Sigma$. Performing $-1/s$ surgery on $C$
  is equivalent to performing $-s$ Dehn twists on an annular
  neighborhood of $C \subset \Sigma$. The bridge sphere $\Sigma$ is
  then a $p$-bridge sphere for $K_s$, showing that
  \[ b_0(K_s) \leq p, \] 
  for any integer $s$. Observe also that
  $\Sigma$ is a $(p+1)$-bridge sphere for $L_s$ with $|\Sigma \cap L_s| = 2p +
  2$.  See~\autoref*{fig:ttk-link}.

  \begin{figure}[h!tb]
    \begin{center}
      \def\svgwidth{0.3\textwidth}
      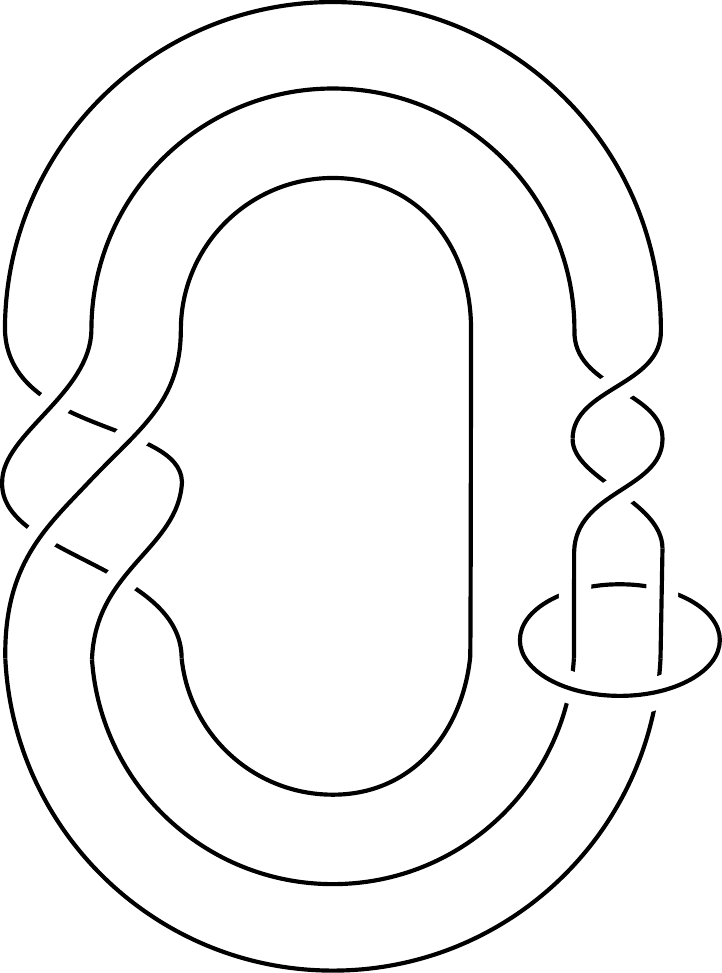
    \end{center}
    \caption{An example of a link $L_s$.  There is a horizontal plane
      which is a bridge sphere $\Sigma$ for $K_s$ containing $C_s$;
      by perturbing $C_s$ slightly, we see that $|\Sigma\cap L_s| = 2p
      + 2$.}
    \label{fig:ttk-link}
  \end{figure}

  By~\autoref*{hyperbolic-twists}, the exterior $M_{L_s} = M_L$ is
  hyperbolic, and thus we can apply Theorem 1.1 of \cite{Lee07}, which
  asserts that $K_s$ is hyperbolic whenever $|s| > 5$.  Suppose $|s| >
  5$.  In this case, the exterior of $K_s$ is then irreducible,
  $\boundary$-irreducible, and anannular.
  
  Let $\Sigma_s$ be a minimal bridge sphere for $K_s$ minimizing the
  pair $(|\Sigma_s \cap K_s|, |\Sigma_s \cap C_s|)$
  lexicographically, so that $\Sigma_s$ is not perturbed. If $\Sigma_s$ is cancellable, then $C_s$ can be
  isotoped onto $\Sigma_s$, as $K_s$ is not the unknot.  As before,
  performing $-1/s$ surgery on $C_s$ is equivalent to performing
  $-s$ twists along $C_s \subset \Sigma_s$. After twisting, we see
  that $\Sigma_s$ is a bridge sphere for $K$ and so $p = b_0(K) \leq
  b_0(K_s)$. Hence, if $\Sigma_s$ is cancellable, $b_0(K_s) = p$.
 
  Suppose, therefore, that $\Sigma_s$ is not cancellable.
  By~\autoref*{weakly-incompressible}, there is a strongly irreducible
  surface $\Sigma_s'$ such that $\Sigma_s'$ satisfies $2 \leq
  |\Sigma_s' \cap C_s| \leq |\Sigma_s \cap C_s|$ and $|\Sigma_s' \cap
  K_s| \leq |\Sigma_s \cap K_s| \leq 2p$, intersects $G$ in a
  non-empty collection of arcs essential on both surfaces, and
  minimizes $|\Sigma'_s \cap G \cap \pd M_L|$ up to isotopy of
  $\Sigma'_s$.  Each component of $\Sigma_s' \cap T_C$ intersects each
  component of $G \cap T_C$ in exactly $|s|$ points, and since $|G
  \cap T_C| = 2$, it follows that
  \[|\Sigma_s' \cap G \cap T_C| = 2|s| \cdot |\Sigma_s' \cap C_s|.\]
  Thus, $\Sigma_s' \cap G$ contains at least $|s| \cdot |\Sigma_s'
  \cap C_s|$ arcs of intersection.  By~\autoref*{arc-annulus-bound},
  \begin{align*}
    |s| \cdot |\Sigma_s' \cap C_s|&\leq 9\chi(\Sigma_{s}' \cap M_{L})\chi(G)\\
    & \leq 18(|\Sigma'_s \cap L_s| - 2) \\
    & \leq 18(|\Sigma'_s \cap K_s| + |\Sigma'_s \cap C_s| - 2) \\
    & \leq 36p + 18|\Sigma'_s \cap C_s| - 36.
  \end{align*} 

  Consequently,
  \[ |s| \leq \frac{36p + 18|\Sigma'_s \cap C_s| - 36}{|\Sigma'_s \cap
    C_s|} = \frac{36p - 36}{|\Sigma_s' \cap C_s|} + 18 \leq 18p. \]
We conclude that if $|s| > 18p$, then $\Sigma_s$ is cancellable and $b_0(K_s) = p$, as desired.
\end{proof}

\section{Construction of the knots}\label{section:construction}
In this section, we construct collections $\set{K^n}$ of twisted torus
knots by twisting one curve $\alpha$ about another curve $\beta$ on a
genus two Heegaard surface $\Sigma$ for $S^3$.  We establish
properties of an associated link exterior in order to use the
machinery developed in~\cite{BakerGordonLuecke13}.  In~\autoref{section:main} we will
use this machinery to show that the knots $\set{K^n}$ have unbounded
genus one bridge number.


The knot $K =T(p,q,r,s)$ has a natural position on $\Sigma$; hence
$b_2(K) = 1$.  See~\autoref*{fig:ttk-genus2} (a picture of the knotted
component of~\autoref*{fig:ttk-link}).

\begin{figure}[h!tb]
  \begin{center}
    \def\svgwidth{0.6\textwidth}
    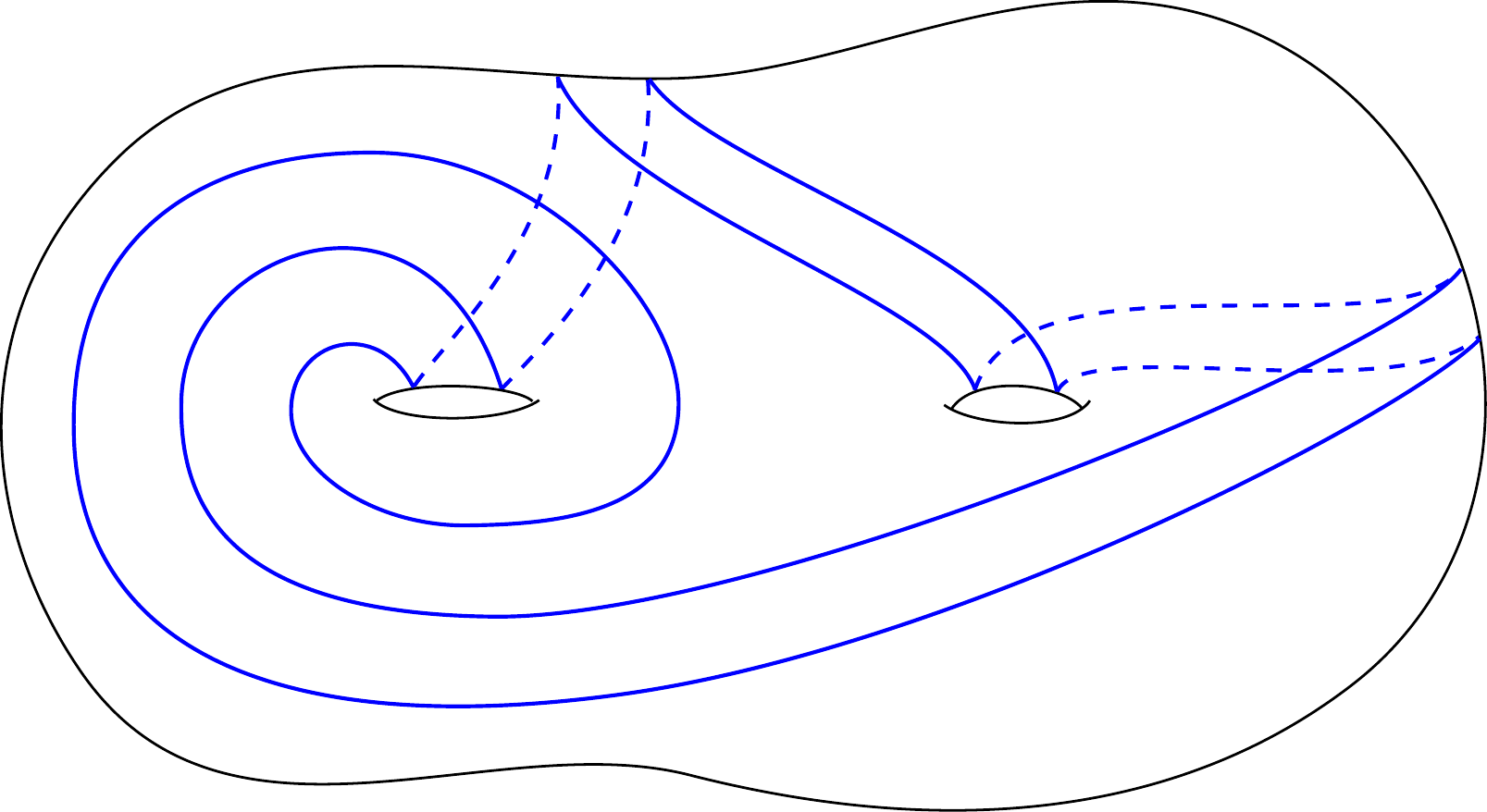
  \end{center}
  \caption{The twisted torus knot $T(3, 2, 2, 1)$ and its position on
    a genus two splitting of $S^3$.}
  \label{fig:ttk-genus2}
\end{figure}

Observe that $K$ meets an obvious disk system for the inside
handlebody $|p|$ and $|r|$ times, respectively, and meets an obvious
disk system for the outside handlebody $|q|$ and $|rs|$ times,
respectively.  Define the \defn{surface slope} of $K$ with respect to
$\Sigma$ to be the isotopy class of $\boundary N(K)\cap \Sigma$ in
$\pd N(K)$. Dean computed the surface slope of $T(p, q, r, s)$:

\begin{lemma}[\cite{Dean03}, Proposition 3.1]\label{sfce-slope}
  The surface slope of $T(p,q,r,s)$ is $pq+r^2s$.
\end{lemma}

Now consider two curves $\alpha$ and $\beta$ on a genus two Heegaard
surface $\Sigma$ for $S^3$.  Choose $\A$ to be a knot of type $T(a,b,0,0)$, without additional restrictions.  This definition requires that $\gcd(a,b)=1$, but $\A$ may be unknotted in $S^3$.  Choose $\beta$ to be a hyperbolic knot or a nontrivial torus knot of type $T(c,d,e,f)$, where $\gcd(c,d)=1$, $1<e<c<d$, $f\neq 0$, and $|ad-bc| \neq 0$.  If $\beta$ is a torus knot of type $(x,y)$, we require two additional properties:

\begin{itemize}
  \item the knot types of $\alpha$ and $\beta$ are different when $\alpha$ and $\beta$ are considered as knots in
    $S^3$, and 
  \item the surface slope of $\beta$ with respect to $\Sigma$ is
    different from $xy$.
\end{itemize}

These requirements may seem burdensome, but they are not difficult to
satisfy in practice.  For example, the knots $T(c, d, 1, f)$, $f\neq
0$, are torus knots of type $(c, d)$ whose surface slope differs from
$cd$ by~\autoref*{sfce-slope}.  Our construction of Berge and Dean
knots below uses this form for $\beta$.
Moreover,~\autoref*{hyperbolic-twists} shows that most twisted torus
knots are hyperbolic, so we do not need to worry about knot types or
surface slopes in this case.

Note that after orienting $\alpha$ and $\beta$, we may isotope them to meet in $\Delta = |ad-bc|$ points of the same intersection sign.  See~\autoref*{fig:alpha-beta} for an example in which $\alpha$ is $T(1, 1, 0, 0)$ and $\beta$ is $T(2, 3, 1, 1)$.

\begin{figure}[h!tb]
  \begin{center}
    \def\svgwidth{0.7\textwidth}
    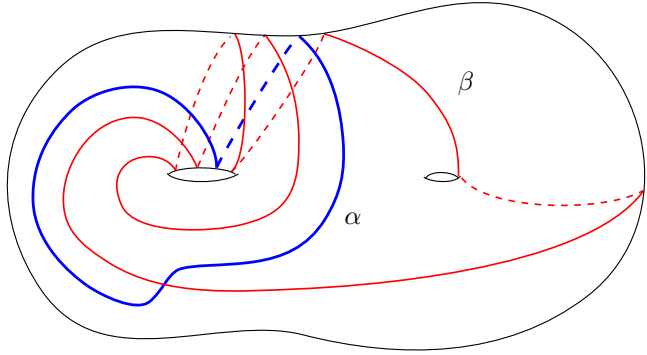
  \end{center}
  \caption{The curves $\alpha$ and $\beta$ on $\Sigma$.  Here $(a,b)=(1,1)$, $(c,d,e,f)=(2,3,1,1)$.}
  \label{fig:alpha-beta}
\end{figure}

Dehn twisting $\alpha$ along $\beta$ results in
the knot $T(a + n\Delta c, b + n\Delta d, n\Delta e, f)$ where $n$ is the
number of twists.  If we twist $\alpha$ along $\beta$ to the left $n$
times in~\autoref*{fig:alpha-beta}, we obtain the knot $T(2n+1, 3n+2,
n, 1)$ since in this case $\Delta=1$.  

Now let $K=\alpha$, and let $L_1$ and $L_2$ be two copies of $\beta$
pushed along the positive and negative normal directions,
respectively, of $\Sigma$.  Then $L_1\cup L_2$ bounds an annulus
$\Rhat$ which meets $K$ exactly $\Delta$ times. Let $\L$ denote the
link $L_1\cup L_2\cup K$, and let $T_1$, $T_2$, and $T_K$ be the
boundary components of $M_{\L}$ arising from $L_1$, $L_2$, and $K$,
respectively.  Let $R=\Rhat\cap M_{L_1\cup L_2}$ and note that the
slope of $R$ on $T_1$ and $T_2$ is the surface slope of $\Sigma$ on
$\beta$.

By orienting $M$ and $\Rhat$, we obtain orientations on $L_1$, $L_2$,
and their meridians $\mu_1$ and $\mu_2$.  With respect to coordinates
on $T_1$ and $T_2$ given by $\mu_1$, $\mu_2$, and $\boundary R$,
performing $1/n$ surgery on $L_1$ and $-1/n$ surgery on $L_2$ has the
effect of ``twisting $K$ along the annulus $\Rhat$,'' an operation
described in~\cite[Definition 1.1]{BakerGordonLuecke13}.  The resulting knot
$K^n$ is the same knot obtained by twisting $\alpha$ along $\beta$ in
the appropriate direction $n$ times.

We first show that for large enough $n$, the knots $K^n$ satisfy the
hypotheses of~\autoref{yoav}.

\begin{lemma}\label{parameter-inequalities}
  Let $p=a + n\Delta c$, $q=b + n\Delta d$, $r=n\Delta e$, and $s=f$ so
  that $K^n=T(p,q,r,s)$.  For sufficiently large $n$,
  \[ 1 < r < p < q. \]
\end{lemma}
\begin{proof}
  Clearly $r>1$ when $n$ is large.  We also have
  \begin{align*}
    a + n\Delta c &= a + n\Delta d + n\Delta (c-d) \\
    &< b + n\Delta d
  \end{align*}
  for large $n$ since $c-d<0$.

  In addition,
  \begin{align*}
    n\Delta e &= n\Delta c + n\Delta (e-c)\\
    &< a + n\Delta c
  \end{align*}
  for large $n$ since $e-c<0$.
\end{proof}

Next, we exhibit a \defn{catching surface} for $(\Rhat, K)$.  The
precise definition is given in~\cite{BakerGordonLuecke13}.  For our purposes an
orientable, connected, properly embedded surface $Q\subset M_{\L}$
catches $(\Rhat, K)$ if $\chi(Q)<0$, $\boundary Q\cap T_i$ is a
nonempty collection of coherently oriented parallel curves on $T_i$,
and $\boundary Q$ intersects $T_i$ in slopes different from $\boundary
R$, $i=1,2$.

\begin{lemma}\label{catching-sfce}
  The pair $(\Rhat, K)$ is caught by a surface $Q$ with
  \[ \chi(Q) = 1-(\abs{a}-1)(\abs{b}-1)-\abs{bc}-\abs{ad}. \]
  Furthermore, $\boundary Q$ is meridional on $T_1$ and $T_2$ and
  meets a meridian of $T_K$ exactly once.
\end{lemma}
\begin{proof}
  We may consider $K$ as lying in a Heegaard torus $T$ for $S^3$.  The
  standard Seifert surface $Q$ for $K$ can be constructed by taking
  $\abs{a}$ disks on one side of $T$, $\abs{b}$ disks on the other
  side, and banding them together with $\abs{ab}$ bands.  Such a
  surface has genus $\frac{1}{2}(\abs{a}-1)(\abs{b}-1)$, so we must
  determine how many times $L_1$ and $L_2$ meet $Q$.

  Observe that the value of $f$ has no effect on $|L_i \cap Q|$ for
  $i=1,2$.  Thus we need only compute $|L_i \cap Q|$ when $f = 0$, and
  in this case we consider $L_i$ to be pushoffs of $(c,d)$ curves
  contained in $T$.

  We may arrange that $L_1$ and $L_2$ meet only the disks on their
  respective sides, and all with the same sign of intersection.  On
  one side, we see that $L_1$ meets $\abs{a}$ disks $\abs{d}$ times
  each.  On the other side, $L_2$ meets $\abs{b}$ disks $\abs{c}$
  times each.  This gives the claimed Euler characteristic, and
  $\boundary Q$ is meridional on $T_1$ and $T_2$.
\end{proof}

An example (with $K$ the unknot) appears
in~\autoref*{fig:linking}.  Performing $-1$ surgery on the unknotted
curve at right and twisting along the annulus bounded by $L_1\cup L_2$
gives the knots in~\autoref*{fig:alpha-beta}.  In this case, the
catching surface is a planar surface with one longitudinal boundary
component on $T_K$ and 5 meridional boundary components on $T_1\cup
T_2$.

\begin{figure}[h!tb]
  \begin{center}
    \def\svgwidth{0.9\textwidth}
    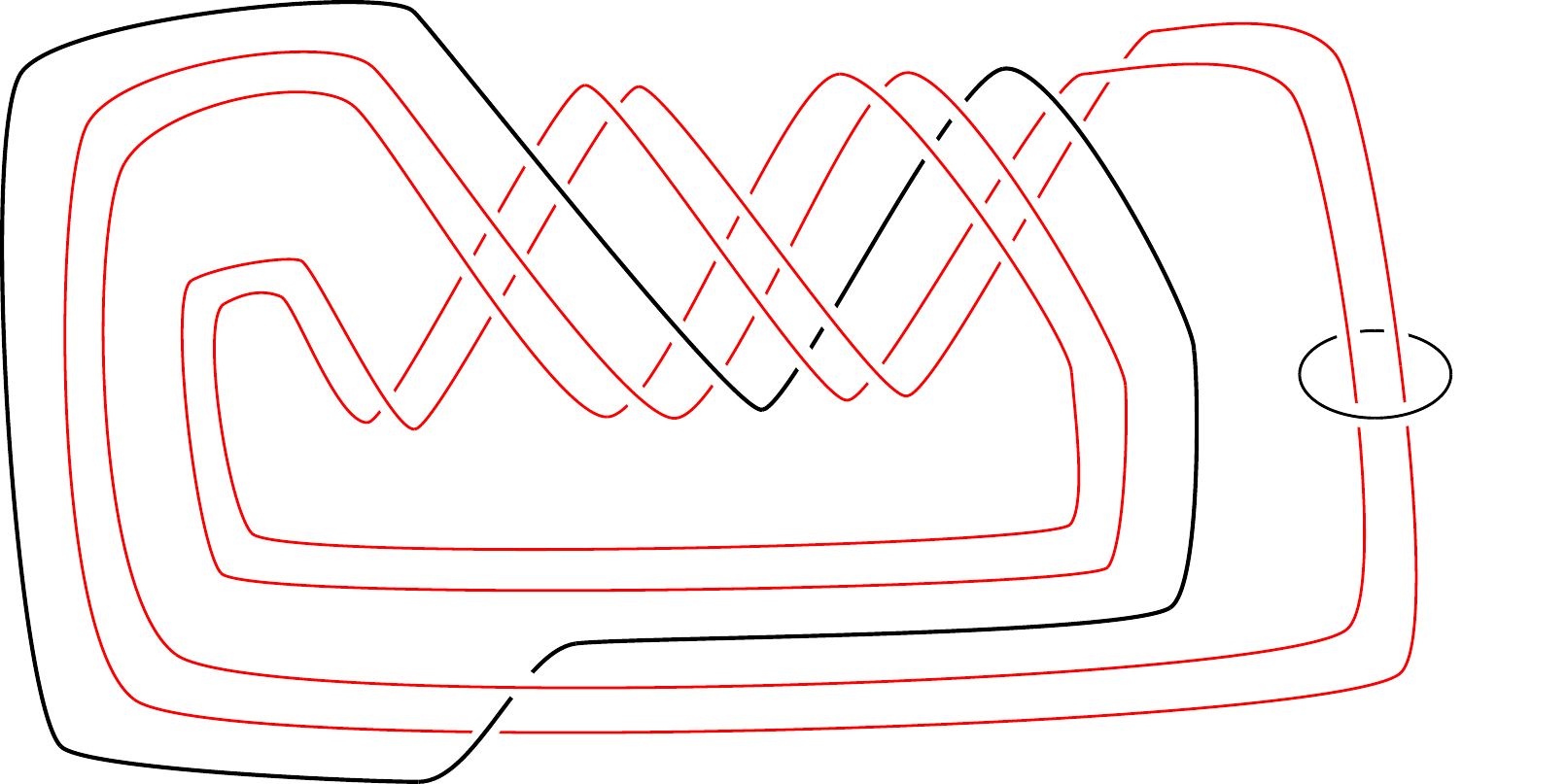
  \end{center}
  \caption{The link $K\cup L_1\cup L_2$ corresponding
    to~\autoref*{fig:alpha-beta}.}
  \label{fig:linking}
\end{figure}

The remainder of this section is devoted to establishing several
topological properties about $\Rhat$ and the manifolds $M_{L_1 \cup
  L_2}$ and $M_{\L}$ in order to employ the tools of~\cite{BakerGordonLuecke13}.

\begin{lemma}\label{annulus-not-isotopic}
  The annulus $\Rhat$ is not isotopic into a genus one splitting of
  $S^3$.
\end{lemma}
\begin{proof}
  We will show that $\beta$, the core of $\Rhat$, is not isotopic into a
  genus one splitting of $S^3$ such that the surface slope of $\beta$
  with respect to this splitting matches the surface slope of $\beta$
  with respect to $\Sigma$.

  Recall that $\beta$ is knotted by hypothesis.  If $\beta$ is a torus
  knot of type $(x,y)$, then our additional hypothesis that $xy$
  differ from the surface slope of $\beta$ with respect to $\Sigma$
  ensures that the surface slopes differ.

  If $\beta$ is not a torus knot, then the core of $\Rhat$ is not
  isotopic into a genus one splitting for $S^3$, and so the conclusion
  certainly holds.
\end{proof}

\begin{lemma}\label{M-irreducible}
  The spaces $M_{L_1\cup L_2}$ and $M_{\L}$ are irreducible.  
\end{lemma}
\begin{proof}
  From the proof of~\autoref*{catching-sfce} we see that the linking
  numbers of $K$ with $L_1$ and $L_2$ are nonzero.  Therefore $M_{\L}$
  is irreducible.

  If $M_{L_1\cup L_2}$ is reducible, a reducing sphere must intersect the annulus $R$ in an essential curve.  However, this shows that $R$ is compressible and therefore $L_1$ or $L_2$ is
  trivial.  Since this is not true, $M_{L_1\cup L_2}$ is irreducible.
\end{proof}

The annulus $R\subset M_{L_1\cup L_2}$ has one boundary component on
each of $T_1$ and $T_2$.  Let $R'=R\cap M_{\L}$.  This is a planar
surface with one longitudinal boundary component on $T_1$, one on
$T_2$, and $|ad-bc|$ coherently oriented meridional boundary
components on $T_K$.

\begin{lemma}\label{no-essential-annulus}
  There is no essential annulus $A\subset M_{\L}$ with one boundary
  component on $T_1$ and the other on $T_2$.
\end{lemma}
\begin{proof}
  Suppose for contradiction that $A$ is such an annulus.  If the
  slopes of $\boundary A$ and $\boundary R$ agree on $T_1$ and $T_2$,
  then we may assume that $\boundary A\cap \boundary R'=\emptyset$
  after an isotopy.  In $S^3$, gluing $A$ to $R$ along subannuli of
  $\pd T_1$ and $\pd T_2$ yields an immersed torus $T$, where $T$
  intersects $\A$ transversely in $|ad-bc| > 0$ points of the same
  orientation.  However, this contradicts that $T$ is homologically
  trivial in $S^3$; hence, there is no such annulus $A$.

  We see that the slopes of $\boundary A$ and $\boundary R$ cannot
  agree.  Note that $A$ is also an essential annulus in $M_{L_1\cup
    L_2}$ with one boundary component on $T_1$ and the other on $T_2$.
  We will show that $M_{L_1\cup L_2}$ contains no such essential
  annuli whose boundary slopes differ from $\boundary R$.
  
  Suppose that it does, and choose an $A \subset M_{L_1\cup L_2}$
  which minimizes $|A \cap R|$.  A cut and paste argument shows that
  $A\cap R$ consists of arcs essential in both $A$ and $R$.  Therefore
  the slopes of $\boundary A$ differ from those of $\boundary R$ on
  both $T_1$ and $T_2$, and these slopes meet each other in points of
  the same intersection sign.  Let $M' = \overline{M_{L_1
      \cup L_2} \setminus N(R)}$, so that $M'$ is homeomorphic to
  $M_{\beta}$, and let $A_1$ and $A_2$ be the two components of the
  frontier of $N(R)$ in $M_{L_1 \cup L_2}$; thus $A_i \subset \pd M'$.
  Since $A \cap R$ is a collection of essential arcs oriented in the
  same direction, $A \cap M' = \overline{A \setminus N(R)}$ is a
  collection of disks, each of which intersects $A_1$ and $A_2$ once.
  This implies that each disk component of $A \cap M'$ has essential
  boundary, contradicting the incompressibility of $\pd M'$.
\end{proof}

\section{Genus one bridge number bounds and proof of the main theorems}\label{section:main}
In this section, we give a lower bound on the genus one bridge numbers
of the knots $K_n$, following which we prove the main theorems of the
paper.  As mentioned above, we will utilize a theorem from~\cite{BakerGordonLuecke13}.
Here we state a version of the theorem specialized to our needs.

\begin{theorem}[\cite{BakerGordonLuecke13}, Theorem 1.2]\label{BGL-maintheorem}
  Let $\L = K\cup L_1\cup
  L_2$ be a link in $S^3$, and let $\Rhat$ be an annulus in $M = S^3$ with
  $\boundary\Rhat=L_1\cup L_2$.  Assume $(\Rhat, K)$ is caught by the
  surface $Q$ in $M_{\L}$ with $\chi(Q)<0$.  Let
  $T_K$, $T_1$, and $T_2$ be the components of $\boundary M_{\L}$
  corresponding to $K$, $L_1$, and $L_2$, respectively.  Suppose that
  $\boundary Q$ is meridional on $T_1$, and $T_2$ and meets a meridian
  of $T_K$ exactly once.  Let $K^n$ be $K$ twisted $n$ times along
  $\Rhat$.  If $H_1\cup_{\Sigma} H_2$ is a genus $g$ Heegaard
  splitting of $S^3$, then either
  \begin{enumerate}
    \item $\Rhat$ can be isotoped to lie in $\Sigma$,
    \item there is an essential annulus properly embedded in $M_{\L}$ with
      one boundary component in each of $T_1$ and $T_2$, or
    \item for each $n$,
      \[ b_g(K^n) \geq \frac{1}{2}\left(
        \frac{n}{-36\chi(Q)}-2g+1\right). \]
  \end{enumerate}
\end{theorem}

Now we bound the genus one bridge number of the knots constructed in
the previous section.  Recall that $K^n$ is the twisted torus knot
$T(a+ n\Delta c, b + n\Delta d, n\Delta e, f)$, where $\Delta$ is the
intersection number of $\alpha$ and $\beta$ in $\Sigma$.

\begin{prop}\label{lower-bound}
  For each $n$ we have
  \[ b_1(K^n) \geq \frac{1}{2}\left(
      \frac{n}{36(\abs{ad}+\abs{bc}+
                    (\abs{a}-1)(\abs{b}-1)-1)}-1\right) \]
\end{prop}
\begin{proof}
  This follows from~\autoref*{BGL-maintheorem}.  The pair $(\Rhat, K)$
  is caught by~\autoref*{catching-sfce}.  The twisting annulus $\Rhat$
  is not isotopic into a genus one Heegaard splitting of $S^3$
  by~\autoref*{annulus-not-isotopic}.  Finally,
  by~\autoref*{no-essential-annulus} there is no essential annulus as
  in case 2.  Therefore, the stated bound holds.
\end{proof}

In order to compute a bound on $b_0(K^n)-b_1(K^n)$, we need an upper
bound on $b_1(K^n)$.  Recall that for large enough $n$, $K^n$ is a $T(p,q,r,s)$ twisted
torus knot with $1<r<p<q$, and consider the Heegaard torus $\Sigma_1$
for $S^3$ depicted in~\autoref*{fig:b_1-bounded-above}.  Here we have
cut the torus along a meridian disk and arranged the parts vertically.
After pulling $p-r$ strands of the pictured twisted torus braid
through $\Sigma_1$ in the first case, or $r$ strands in the second
case, $\Sigma_1$ becomes a bridge surface for $T(p,q,r,s)$.  This
gives the upper bound of the following lemma.

\begin{figure}[h!tb]
  \begin{center}
    \def\svgwidth{1.0\textwidth}
    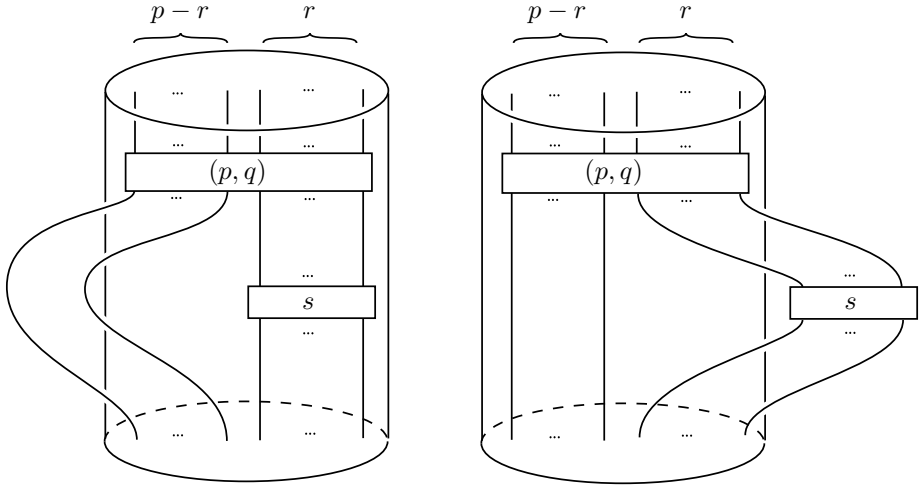
  \end{center}
  \caption{A bridge torus for $T(p,q,r,s)$.  The top box is a $(p,q)$
    torus braid, and the lower box represents $s$ full twists.}
  \label{fig:b_1-bounded-above}
\end{figure}

\begin{lemma}\label{b_1-upper-bound}
  We have
  \[ b_1(T(p,q,r,s)) \leq \min\{r, p-r\}. \]
\end{lemma}

We need the next lemma, which is implicit in the work of
Dean~\cite{Dean03} and appears explicitly in Lemma 3.2 of~\cite{Lee13},
to prove that the knots in~\autoref*{thm:hyperbolic-gaps} may be taken
to have tunnel number one.  Later, we will use it to give examples of
Berge knots with arbitrarily large genus one bridge number.  Let the
genus $2$ splitting $\Sigma$ bound handlebodies $H_1$ and $H_2$ in
$S^3$, and think of $H_1$ as the ``inside'' handlebody
in~\autoref*{fig:alpha-beta}.

\begin{lemma}\label{ttk-primitive}
  The knot $T(p,q,r,s)$ (with notation as above and $p,q>1$) is
  primitive in $H_1$ if and only if $r\equiv \pm 1\pmod{p}$ or $r\equiv \pm q
  \pmod{p}$.  This knot is primitive in $H_2$ if and only if
  \begin{enumerate}
  \item $s=\pm 1$, and
  \item $r\equiv \pm1 \pmod{q}$ or $r\equiv \pm p\pmod{q}$.
  \end{enumerate}
\end{lemma}

We can now
prove~\autoref*{thm:hyperbolic-gaps}.

\begin{proof}[Proof of~\autoref*{thm:hyperbolic-gaps}]
  Choose curves $\alpha$ and $\beta$ as
  in~\autoref*{section:construction} and let $K^n=T(p,q,r,s)$ be the
  knot defined by twisting $\alpha$ along $\beta$ in the genus 2
  splitting $\Sigma$.  Given $C>0$, we may choose $n$ large enough so
  that~\autoref*{parameter-inequalities} is satisfied and
  $b_1(K^n)\geq C$ by~\autoref*{lower-bound}.  Fix this $n$ (so that
  $p$, $q$, and $r$ are also fixed) and note that $s$ does not appear
  in the bound for $b_1(K^n)$ given by~\autoref*{lower-bound}.
  By~\autoref*{hyperbolic-twists} and~\autoref*{yoav} the knots
  $K_s=T(p,q,r,s)$ are hyperbolic and have $b_0(K_s) = p$ for
  $|s|>18p$.  Furthermore, $b_1(K_s)\leq \min(r, p-r) \leq 
  \frac{1}{2}p$ by~\autoref*{b_1-upper-bound}.

  To see that we may choose such examples to have tunnel number one,
  note that we may choose $\alpha$ and $\beta$ so that the knots $K^n$
  are primitive in $H_1$ according to~\autoref*{ttk-primitive}.  For
  example, choosing $\alpha$ to be $T(1, 1, 0, 0)$ and $\beta$ to be
  $T(m, m+1, 1, s)$ for $m>1$ results in the family $K^n=T(mn+1,
  mn+n+1, n, s)$ (cf. proof of~\autoref*{thm:hyperbolic-berge}).  A
  knot which is primitive on one side of $\Sigma$ is isotopic to a
  core of the handlebody on that side, and so has tunnel number one.
\end{proof}

Now we examine special cases of twisted torus knots. 

\begin{proof}[Proof of~\autoref*{thm:hyperbolic-berge}]
  Recall the curves $\alpha$ and $\beta$ of~\autoref*{fig:alpha-beta}.
  Twisting $n$ times around $\beta$ we obtain the knots $T(3n+ 1, 2n+
  1, n, 1)$, which are Berge knots by~\autoref*{ttk-primitive}.

  More generally, let $\alpha$ be a $T(1, 1, 0, 0)$ twisted torus knot
  and $\beta$ be a $T(m, m+1, 1, \pm 1)$ twisted torus knot for $m>1$.
  Since $\Delta = 1$, we obtain the knot $K^n=T(mn + 1, mn+n+1, n, \pm
  1)$.  Letting $p=mn + 1$, $q=mn+n+1$, $r=n$, and $s=\pm 1$, an easy
  calculation shows that $K^n$ is primitive on both sides of $\Sigma$,
  so these are Berge knots.

  Applying the bound of~\autoref*{lower-bound} to these knots, we see
  that they have arbitrarily large genus one bridge number.  We use the
  argument of Proposition 4.3 from~\cite{MoriahSedgwick09} to show that
  these knots are hyperbolic: Because $K^n$ is primitive on at least
  one side of $\Sigma$, it is a tunnel number one knot. Toroidal
  tunnel one knots are classified by Morimoto and
  Sakuma~\cite{MorimotoSakuma91} and have genus one bridge number one.
  Therefore, for large enough $n$, $K^n$ is atoroidal and not a torus
  knot; thus it is hyperbolic.

  To see that these are Berge knots of type VII and VIII (knots which
  lie in the fiber of a trefoil or figure eight knot,
  see~\cite{Baker08}), we note that $K^n$ lies in a neighborhood of
  $\alpha\cup\beta$ in $\Sigma$ and examine the boundary of this
  neighborhood in $S^3$.  Introduce an unknotted curve $c$ with
  surgery coefficient $\mp 1$ according to whether $s=\pm 1$ as
  in~\autoref*{fig:Berge-types}.  Before surgery on $c$, we see
  $N(\alpha\cup\beta)$ as the neighborhood of a $a=(1, 1)$ curve and a
  $b_m = (m, m+1)$ curve on a genus one splitting $T$ of $S^3$.  It is
  not difficult to see that $\pd N(a\cup b_1)$ becomes the figure 8 knot
  after $+1$ surgery on $c$ and a trefoil knot after $-1$ surgery.
  But $N(a\cup b_m)$ is isotopic to $N(a\cup b_1)$ in $T$ since we can
  ``untwist'' along $a$.  Therefore $N(a\cup b_m)$ is a genus one Seifert surface for $\pd N(a \cup b_m)$, which becomes a trefoil or figure eight after the appropriate surgery on $c$, and
  so $K^n$ is a Berge knot of type VII or VIII.
\end{proof}

\begin{figure}[h!tb]
  \begin{center}
    \def\svgwidth{0.4\textwidth}
    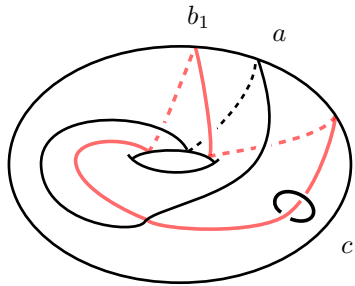
  \end{center}
  \caption{The curves $a$ and $b_1$ lying on a genus 1 splitting
    of $S^3$, together with the twisting curve $c$}
  \label{fig:Berge-types}
\end{figure}

Next, we consider the knots studied by Dean in~\cite{Dean03}. These are
knots which lie in a genus two splitting of $S^3$ so that they are
primitive on one side and Seifert fibered on the other (in the sense that attaching
a $2$--handle to the handlebody along the knot results in a Seifert
fibered space).  Such
knots have small Seifert fibered surgeries, so in this sense Dean
knots are a generalization of Berge knots.  Dean shows that the knots $T(p, q, 2q-p, \pm 1)$, $(p+1)/2<q<p$, and
$T(p, q, p-kq, \pm 1)$, $1<q<p/2$, $2\leq k\leq (p-2)/q$ are
primitive/Seifert fibered \cite[Theorem
4.1]{Dean03}.  

\begin{proof}[Proof of~\autoref*{thm:hyperbolic-dean}]
  Let $\alpha$ be $T(2, 1, 0, 0)$ and $\beta$ be $T(2m-1, m, 1, \pm
  1)$.  Twisting $\alpha$ around $\beta$ yields $T(2mn -n + 2, mn +1,
  n, \pm 1)$.  With $p=2mn-n+2$, $q=mn+1$, $r=n$, and $s=\pm 1$, we
  see that $r = 2q-p$ and $\frac{p+1}{2}<q<p$ for large enough $n$.
  These are Dean knots of the first type in~\cite[Theorem 4.1]{Dean03}.

  Similarly, if we let $\alpha$ be $T(l, 1, 0, 0)$ and $\beta$ be
  $T(lm+1, m, 1, \pm 1)$ for $l\geq 2$, $m \geq 2$ we obtain the knot
  $T(p,q,r,s)$ with $p=(lm+1)n+l$, $q=mn+1$, $r=n$, and $s=\pm 1$
  after twisting. It is clear that $r=p-lq$, $1<q<p/2$, and $2\leq
  l\leq\frac{p-2}{q}$ for large enough $n$, and so these are Dean
  knots of the second type in~\cite[Theorem 4.1]{Dean03}.

  The same arguments used in the proof
  of~\autoref*{thm:hyperbolic-berge} apply to show that these knots
  are hyperbolic and have arbitrarily large genus one bridge number.
\end{proof}

Finally, we show that the knots from~\autoref*{thm:hyperbolic-berge} have minimal genus Heegaard splittings of Hempel distance two.

\begin{proof}[Proof of~\autoref*{cor:distance-two}]
  As in the proof of~\autoref*{thm:hyperbolic-berge}, for sufficiently
  large $n$ the knots $K^n = T(mn+1,mn+n+1,n,\pm 1)$ are hyperbolic
  Berge knots with unbounded genus one bridge number.  Since each
  $K^n$ has a doubly primitive representative on a genus two Heegaard
  surface $\Sigma'$ for $S^3$, there are compressing disks $D$ and
  $D'$ on opposite sides of $\Sigma'$ which are disjoint from $K^n$.
  Pushing $K^n$ off of $\Sigma'$ into one of the handlebodies yields a
  genus two Heegaard surface $\Sigma$ for $M_{K^n}$, where $D$ and
  $D'$ are compressing disks for $\Sigma$ and the distance between
  $\boundary D$ and $\boundary D'$ in the curve complex of $\Sigma$ is
  at most 2.  On the other hand, the distance of $\Sigma$ cannot be
  zero or one by~\cite{CassonGordon87} because $K^n$ is hyperbolic.
  Therefore, the distance of $\Sigma$ is exactly two.
\end{proof}









\bibliographystyle{plain}
\bibliography{spectrum}
\end{document}